\documentclass[11pt, a4paper, reqno]{amsart}
\usepackage[a4paper, margin=3.75cm]{geometry}
\usepackage[T2A]{fontenc}
\usepackage[cp866]{inputenc}
\usepackage[russian,english]{babel}
\usepackage{amsfonts,amssymb,mathrsfs,amscd,comment}
\usepackage{amsthm,graphicx} 
\usepackage{datetime}
\def\NoBlackBoxes{\overfullrule0pt}

\NoBlackBoxes
\theoremstyle{plain}
\newtheorem{theorem}{Theorem}

\theoremstyle{definition}

\usepackage[breaklinks]{hyperref}
\hypersetup{
unicode=true,
colorlinks=true,
linkcolor=blue,
citecolor=blue,
urlcolor=blue,
filecolor=blue,
bookmarksnumbered=true,
pdfstartview=FitH,
pdfhighlight=/N
}
\theoremstyle{main}

\let\savedef=\endproof
\def\endproof{~$\square$\savedef}
\def\bad{\spaceskip=0.33emplus0.6emminus0.15em\immediate\write5{\string\bad}}
\let\leq\leqslant

\def\({\left(}
\def\){\right)}
\def\[{\left[}
\def\]{\right]}
\def\<{\left\langle}
\def\>{\right\rangle}

\def\GL{\operatorname{GL}}

\def\mdeg{\operatorname{deg}}

\def\const{\mathrm{const}}

\def\mdeg{\operatorname{deg}}

\def\NN{\mathbb N}

\def\CC{\mathbb C}

\def\nn{\mathbf n}

\def\dd{\mathrm{dd}}

\let\leq\leqslant\let\geq\geqslant
\let\ge\geqslant
\let\myt\widetilde
\def\bad{\spaceskip=0.33emplus0.6emminus0.15em\immediate\write5{\string\bad}}

\def\GL{\mathrm{GL}}
\def\nn{\mathbf n}
\def\dd{\mathbf d}
\def\uu{\mathbf u}
\def\vv{\mathbf v}
\def\ff{\mathbf f}

\def\No{\rm №}
\begin{document}

\selectlanguage{english}

\title{On Some Algebraic Properties of Hermite--Pad\'e Polynomials}
\author[Sergey~P.~Suetin]{Sergey~P.~Suetin}
\address{Steklov Mathematical Institute of the Russian Academy of Sciences, Russia}
\email{suetin@mi-ras.ru}


\maketitle

\markright{Algebraic Properties of Hermite--Pad\'e Polynomials}

\begin{abstract}
Let  $[f_0,\dots,f_m]$ be a tuple of
series in nonnegative powers of $1/z$, $f_j(\infty)\neq0$. It is supposed that the tuple is in ``general position''.  We give a construction of type I and type II Hermite--Pad\'e polynomials to the given tuple of degrees $\leq{n}$ and $\leq{mn}$ respectively and the corresponding $(m+1)$-multi-indexes with the following property. Let $M_1(z)$ and $M_2(z)$ be two $(m+1)\times(m+1)$ polynomial matrices, $M_1(z),M_2(z)\in\GL(m+1,\CC[z])$, generated by type I and type II Hermite--Pad\'e polynomials respectively. Then we have $M_1(z)M_2(z)\equiv I_{m+1}$, where $I_{m+1}$ is the identity $(m+1)\times(m+1)$-matrix.

The result is motivated by some novel applications of Hermite--Pad\'e polynomials  to the investigation of monodromy properties of Fuchsian systems of differential equations; see~\cite{Yam09}, \cite{Man12},
~\cite{NoTsYa13}, \cite{MaTs17},~\cite{Nag21}.

Bibliography:~\cite{Yam09}~titles.

\end{abstract}

\vskip1cm



\subsection*{1. Introduction}\label{s1s1}

The very well-known and very important algebraic relations between type I and type II Hermite--Pad\'e (HP) polynomials were discovered by K.~Mahler (see~\cite{Mah68}). These relations connect to each other type I and type II HP polynomials associated with a tuple of formal power series given at zero point $\zeta=0$. Thus from those relations it follows that type I and type II HP polynomials are not independent. In the last decade based on  G.~Chudnovsky ideas~\cite{Chu80},~\cite{Chu80b} some generalizations of Mahler's identities were obtained and successfully applied to the monodromy problems of Fuchsian systems of differential equations; see~\cite{Yam09},~\cite{Man12}, \cite{NoTsYa13}, \cite{MaTs17},~\cite{NaYa18},~\cite{Nag21} and the bibliography therein.
Similar to Mahler's results, these generalizations were produced for power series given at zero point $\zeta=0$.  But after then while applying the  identities to monodromy problems it was made the changing of the variable from $\zeta$ to $z=1/\zeta$; see~\cite{Yam09} and~\cite{MaTs17}. The reason is that it is namely in such a form the identities are applicable to monodromy problems.

The main purpose of the current paper is to propose a construction of type I and type II HP polynomials associated with a tuple $[f_0,\dots,f_m]$ of formal power series given not at zero point but at the infinity point $z=\infty$. Under some assumption on the ``general position'' of the tuple (cf.~\cite{NiSo88},~\cite{IkSu21},~\cite{StRy21}) this construction  provides two $(m+1)\times(m+1)$ polynomial matrices $M_1(z)$ and $M_2(z)$, $M_1(z),M_2(z)\in\GL(m+1,\CC[z])$, with the following property: $M_1(z)M_2(z)\equiv I_{m+1}$, where $I_{m+1}$ is the identity $(m+1)\times(m+1)$-matrix.

\subsection*{2. The case of $m=1$}\label{s1s2}
To make our idea of the construction of HP polynomials with the prescribed properties, we start with the case when $m=1$, i.e. the case of Pad\'e polynomials. Thus we are given a tuple $[f_0,f_1]$ of two formal power series at infinity point. We assume that the tuple $[f_0,f_1]$ is in a ``general position''. For the case $m=1$ that means that all the indexes of the Pad\'e table for the series $f=f_1/f_0$ are normal (see~\cite{NiSo88}). Evidently it is also the case for the Pad\'e table for the reciprocal series $1/f$.

Let $\nn_0:=(n,n-1)$ and $\nn_1:=(n-1,n)$ be two multi-indexes, $n\geq1$.

It is easy to see that for multi-index $\nn_0$ there exist two polynomials $Q^{(0)}_0$, $\mdeg{Q^{(0)}_0}\leq n$, and $Q^{(0)}_1$, $\mdeg{Q^{(0)}_1}\leq{n-1}$ with the following property\footnote{Here and in what follows we consider relations similar to~\eqref{1} only as formal relations in the space of formal power series.}:
\begin{equation}
Q^{(0)}_0 f_0+zQ^{(0)}_1f_1=O\(\frac1{z^n}\),\quad z\to\infty.
\label{1}
\end{equation}
From relation~\eqref{1} it directly follows that $Q^{(0)}_0(0)\neq0$. Indeed, for otherwise we have that $Q^{(0)}_0=z\myt{Q}^{(0)}_0$, where $\mdeg{\myt{Q}^{(0)}_0}\leq{n-1}$. Therefore from~\eqref{1} it would follow that
\begin{equation}
\myt{Q}^{(0)}_0 +Q^{(0)}_1f=O\(\frac1{z^{n+1}}\),\quad z\to\infty.
\label{2}
\end{equation}
The relation~\eqref{2} implies that the index $(n-1,n-1)$ is not normal for the series $f=f_1/f_0$ which is in contradiction to our assumption. Thus we have that $\mdeg{Q^{(0)}_1}=n-1$, $\mdeg{Q^{(0)}_0}=n$ and $Q^{(0)}_0(0)\neq0$. Therefore we can normalize $Q^{(0)}_0$ as $Q^{(0)}_0(0)=1$

Just in the similar way we obtain that for multi-index $\nn_1$ there exist two polynomials $Q^{(1)}_0$, $\mdeg{Q^{(1)}_0}\leq{n-1}$, and $Q^{(1)}_1$, $\mdeg{Q^{(1)}_1}\leq{n}$, with the property
\begin{equation}
zQ^{(1)}_0f_0+Q^{(1)}_1f_1=\(\frac1{z^n}\),\quad z\to\infty,
\label{3}
\end{equation}
Based on normality of Pad\'e table for the series $1/f$ it follows from~\eqref{3} that $\mdeg{Q^{(1)}_0}=n-1$, $\mdeg{Q^{(1)}_1}=n$, $Q^{(1)}_1(0)\neq0$ and thus we can introduce the normalization $Q^{(1)}_1(0)=1$.

From~\eqref{1} and~\eqref{3} it easy follows the relation
\begin{equation}
Q^{(0)}_0Q^{(1)}-z^2Q^{(0)}_1Q^{(1)}_0=O(1),
\quad z\to\infty.
\label{4}
\end{equation}
Therefore the polynomial
$\mathcal P:=Q^{(0)}_0Q^{(1)}-z^2Q^{(0)}_1Q^{(1)}_0$ is constant. To find this constant we evaluate this polynomial at zero point to obtain that
$\mathcal P(0)=1$. Thus we  obtain that $Q^{(0)}_0Q^{(1)}-z^2Q^{(0)}_1Q^{(1)}_0\equiv1$. Therefore for the $2\times2$-matrix
\begin{equation}
M(z):=
\begin{pmatrix}
Q^{(0)}_0 & zQ^{(0)}_1\\
zQ^{(1)}_0 & Q^{(1)}_1
\end{pmatrix}
\notag
\end{equation}
we have that $\operatorname{det}M(z)\equiv1$ and
\begin{equation}
M^{-1}(z)=
\begin{pmatrix}
Q^{(1)}_1 & -zQ^{(0)}_1\\
-zQ^{(1)}_0 & Q^{(0)}_0
\end{pmatrix},
\notag
\end{equation}
where $\mdeg{Q^{(0)}_0}=\mdeg{Q^{(1)}_1}=n$ (cf.~\cite{Yam09}).

\subsection*{3. The case of general $m$}\label{s1s3}
Let now consider the tuple $[f_0,\dots,f_m]$ of $m\ge2$ series in nonnegative powers of $1/z$, $f_j(\infty)\neq0$, and such that the tuple is in a ``general position''. Here ``general position'' means that all multi-indexes $\nn=(n_0,\dots,n_m)\in\NN^{m+1}$ are normal for the HP polynomials associated with the tuple at the infinity point (see~\cite{NiSo88},~\cite{StRy21}).

Set $n\in\NN$, $\nn_k:=(n-1,\dots,n-1,\underbrace{n}_{k+1},n-1,\dots,n-1)
\in\NN^{m+1}$, $k=0,\dots,m$, be  a multi-index. It is easy to see that for each $k=0,\dots,m$ and the corresponding $\nn_k$ there exist polynomials $Q^{(k)}_j$, $j=0,\dots,m$, $\mdeg{Q^{(k)}_j}\leq{n-1}$, $j\neq k$, $\mdeg{Q^{(k)}_k}\leq{n}$, with the following property
\begin{equation}
zQ^{(k)}_0f_0+\dots+zQ^{(k)}_{k-1}f_{k-1}+
Q^{(k)}_kf_k+zQ^{(k)}_{k+1}+\dots+zQ^{(k)}_mf_m
=O\(\frac1{z^{mn}}\).
\label{5}
\end{equation}
The relation~\eqref{5} implies that $\mdeg{Q^{(k)}_k}=k$ and $Q^{(k)}_k(0)\neq0$ because otherwise we would obtain that multi-index $(n-1,\dots,n-1)\in\NN^{m+1}$ is not normal for the given tuple $[f_0,\dots,f_m]$. Therefore we can normalize $Q^{(k)}_k(0)=1$,
$k=0,\dots,m$.

For each $n\in\NN$ let $\dd_s:=(mn-1,\dots,mn-1,\underbrace{mn}_{s+1},mn-1,\dots,mn-1)\in\NN^{m+1}$ be a multi-index, $s=0,\dots,m$. It is easy to see that for each $s=0,\dots,m$ and the corresponding $\dd_s$ there exist polynomials $P^{(s)}_j$, $j=0,\dots,m$, $\mdeg{P^{(s)}_j}\leq{mn-1}$, $j\neq s$, $\mdeg{P^{(s)}_s}\leq{mn}$, with the following property
\begin{equation}
zf_sP^{(s)}_j-f_jP^{(s)}_s=O\(\frac1{z^n}\),\quad j=0,\dots,m,
\quad j\neq s.
\label{6}
\end{equation}
The relations~\eqref{6} imply that $\mdeg{P^{(s)}_s}=mn$ and $P^{(s)}_s(0)\neq0$ because otherwise we would obtain that the multi-index $(mn-1,\dots,mn-1)$ is not normal for the given tuple $[f_0,\dots,f_m]$. Thus we can normalize $P^{(s)}_s(0)=1$.

The following result is valid.

\begin{theorem}\label{the1}
Let
\begin{align}
M_1(z):&=\bigl(zQ^{(k)}_0,\dots,zQ^{(k)}_{k-1},Q^{(k)}_k,zQ^{(k)}_{k+1},\dots,zQ^{(k)}_m\bigr)_{k=0,\dots,m},\notag\\
M_2(z):&=\bigl(zP^{(s)}_0,\dots,zP^{(s)}_{s-1},P^{(s)}_s,zP^{(s)}_{s+1},\dots,zP^{(s)}_m\bigr)_{s=0,\dots,m}\notag
\end{align}
be two polynomial $(m+1)\times(m+1)$-matrices. Then $M_1(z)M_2(z)\equiv I_{m+1}$, where $I_{m+1}$ is the identity $(m+1)\times(m+1)$-matrix.
\end{theorem}

\begin{proof}
Set
\begin{align}
\uu_k(z):&=\bigl(zQ^{(k)}_0,\dots,zQ^{(k)}_{k-1},Q^{(k)}_k,zQ^{(k)}_{k+1},\dots,zQ^{(k)}_m\bigr),\quad  k=0,\dots,m,\notag\\
\text{and}&\notag\\
\vv_s(z):&=\bigl(zP^{(s)}_0,\dots,zP^{(s)}_{s-1},P^{(s)}_s,zP^{(s)}_{s+1},\dots,zP^{(s)}_m\bigr),\quad s=0,\dots,m,\notag
\end{align}
be two polynomial vectors. Then the scalar product $\mathcal P_{k,s}(z):=\uu_k(z)\cdot \vv_s^{\mathrm T}(z)$ is a polynomial is $z$. Set $\mathbf f:=(f_0,\dots,f_m)$. From~\eqref{5} and~\eqref{6} it follows that
\begin{align}
\mathcal P_{k,s}(z)&=\frac1{f_s}\uu_k(z)\cdot f_s\vv_s^{\mathrm T}
=
\frac1{f_s}\uu_k(z)\cdot \(P^{(s)}_s\ff^{\mathrm T}+O\(\frac1{z^n}\)\)
\notag\\
&=\frac1{f_s}\bigl(P^{(s)}_s\uu_k(z)\cdot\ff^{\mathrm T}+O(1)\bigr)
=\frac1{f_s}\biggl(P^{(s)}_s\cdot O\(\frac1{z^{mn}}\)+O(1)\biggr)=O(1).
\label{7}
\end{align}
From~\eqref{7} we obtain that $\mathcal P_{k,s}(z)=O(1)$ as $z\to\infty$.
Thus polynomial $\mathcal P_{k,s}(z)\equiv\const=\mathcal P_{k,s}(0)$.

We have to consider two different cases.

1) If $k=s$ then
$$
\mathcal P_{k,k}(z)=\uu_k(z)\cdot\vv_k^{\mathrm T}(z)
=Q^{(k)}_kP^{(k)}_k+zp(z).
$$
Thus $\mathcal P_{k,k}(0)=Q^{(k)}_k(0)P^{(k)}_k(0)=1$.

2) If $k\neq s$ then $\mathcal P_{k,s}(z)=zq(z)$ and thus $\mathcal P_{k,s}(0)=0$.

Theorem~\ref{the1} is proved.
\end{proof}

\def\by#1;{#1\unskip,}
\def\paper#1;{``#1\unskip''\unskip,}
\def\paperinfo#1;{#1\unskip.}
\def\eprint#1;{``#1\unskip''\unskip,}
\def\eprintinfo#1;{#1\unskip,}
\def\book#1;{``#1\unskip''\unskip,}
\def\inbook#1;{``#1\unskip''\unskip,}
\def\bookinfo#1;{#1\unskip,}
\def\jour#1;{#1\unskip,}
\def\issue#1;{\No~#1\unskip,}
\def\yr#1;{#1\unskip,}
\def\pages#1.{#1\unskip.}
\def\vol#1;{\textbf{#1}\unskip,}
\def\finalinfo#1;{#1\unskip.}
\def\publ#1;{#1\unskip,}
\def\publadrr#1;{#1\unskip,}
\def\publaddr#1;{#1\unskip,}
\def\procinfo#1;{#1\unskip,}
\def\serial#1;{#1\unskip,}
\def\ed#1;{ed #1\unskip,}
\def\eds#1;{ed #1\unskip,}

\end{document}